\documentclass
{amsart}

\usepackage{amsmath,amsfonts,amssymb,amsthm,xcolor,amscd}
\usepackage{latexsym,xspace,enumerate}
\usepackage[mathscr]{eucal}
\usepackage[all]{xypic}
\usepackage{hyperref}
\usepackage{vmargin}
\setmarginsrb{17mm}{0mm}{17mm}{5mm}%
    {16mm}{7mm}{16mm}{10mm}

\newtheorem{theorem}{Theorem}[section]

\newtheorem{question}[theorem]{Question}
\newtheorem{proposition}[theorem]{Proposition}
\newtheorem{lemma}[theorem]{Lemma}
\newtheorem{fact}[theorem]{Fact}
\newtheorem{corollary}[theorem]{Corollary}
\newtheorem{conjecture}[theorem]{Conjecture}

\theoremstyle{definition}
\newtheorem{definition}[theorem]{Definition}
\newtheorem{example}[theorem]{Example}
\newtheorem{remark}[theorem]{Remark}

\numberwithin{equation}{section}

\newcommand{\R}{\mathbb R}

\newcommand{\N}{\mathbb N}

\newcommand{\Q}{\mathbb Q}
\newcommand{\Z}{\mathbb Z}
\newcommand{\Prm}{\mathbb P}

\def\B{\mathcal B}

\def\End{\mathrm{End}}
\def\Aut{\mathrm{Aut}}

\def\nub{\mathrm{nub}}

\def\U{\mathcal U}

\newcommand{\rank}{\mathrm{rank}}
\newcommand{\rk}{\mathrm{rk}}

\numberwithin{equation}{section}

\title[Finiteness of topological entropy for locally compact abelian groups]{Finiteness of topological entropy for \\ locally compact abelian groups}

\author[D. Dikranjan]{Dikran Dikranjan$^\sharp$}
\address{$^\sharp$ Dipartimento di Scienze Matematiche, Informatiche e Fisiche\endgraf
\ \ Universit\`a di Udine\endgraf
\ \ Via delle Scienze 206, 33100 Udine, Italy\endgraf 
\ \ Email: \texttt{dikran.dikranjan@uniud.it}}

\author[A. Giordano Bruno]{Anna Giordano Bruno$^\flat$}
\address{$^\flat$ Dipartimento di Scienze Matematiche, Informatiche e Fisiche\endgraf
\ \ Universit\`a di Udine\endgraf
\ \ Via delle Scienze 206, 33100 Udine, Italy\endgraf 
\ \ Email: \texttt{anna.giordanobruno@uniud.it}}

\author[F. G. Russo]{Francesco G. Russo$^\natural$}
\address{$^\natural$ Department of Mathematics and Applied Mathematics\endgraf 
\ \ University of Cape Town\endgraf
\ \ Private Bag X1, Rondebosch 7701, Cape Town, South Africa\endgraf
\ \ Email: \texttt{francescog.russo@yahoo.com}}

\date{}

\begin{document}



\maketitle

\begin{abstract}
We study the locally compact abelian groups in the class $\mathfrak E_{<\infty}$, that is, having only continuous endomorphisms of finite topological entropy, and in its subclass $\mathfrak E_0$, that is, having all continuous endomorphisms with vanishing topological entropy.  We discuss the reduction of the problem to the case of periodic locally compact abelian groups, and then to locally compact abelian $p$-groups.

We show that locally compact abelian $p$-groups of finite rank belong to $\mathfrak E_{<\infty}$, and that those of them that belong to $\mathfrak E_0$ are precisely the ones with discrete maximal divisible subgroup.
Furthermore, the topological entropy of endomorphisms of locally compact abelian $p$-groups of finite rank coincides with the logarithm of their scale.

The backbone of the paper is the Addition Theorem for continuous endomorphisms of locally compact abelian groups. 
Various versions of the  Addition Theorem are established in the paper and used in the proofs of the main results, but its validity in the general case remains an open problem. 
\end{abstract}


\section{Introduction}

The topological entropy for continuous self-maps of compact spaces was introduced by Adler, Konheim and McAndrew \cite{AKM}, in analogy with the metric entropy from ergodic theory introduced by Kolmogorov and Sinai (see \cite{Walters}).
Later, Bowen \cite{B} gave a definition of topological entropy for uniformly continuous self-maps of metric spaces, which was extended by Hood \cite{hood} to uniformly continuous self-maps of uniform spaces. This notion of topological entropy coincides with that in \cite{AKM} in the compact case, when the given compact topological space is endowed with the unique uniformity compatible with the topology. 

Hood's definition of topological entropy applies to topological groups $G$, endowed with their left uniformity $\U$, as continuous endomorphisms  $\phi:G\to G$ are uniformly continuous with respect to $\U$. Assume that $G$ is a locally compact group, denote by $\mathcal C(G)$ the family of all compact neighborhoods of $1$ in $G$, and let $\mu$ be a left Haar measure on $G$. For $U\in\mathcal C(G)$ and $n\in\N_+$, the \emph{$n$-th $\phi$-cotrajectory of $U$} is $$C_n(\phi,U)=U\cap\phi^{-1}(U)\cap\ldots\cap\phi^{-n+1}(U)\in\mathcal C(G).$$
The \emph{topological entropy} of $\phi$ with respect to $U\in\mathcal C(G)$ is 
$$H_{top}(\phi,U)=\limsup_{n\to\infty}\frac{-\log\mu(C_n(\phi,U))}{n}$$
and the \emph{topological entropy} of $\phi$ is $$h_{top}(\phi)=\sup\{H_{top}(\phi,U)\mid U\in\mathcal C(G)\}.$$
 Following \cite{DG-islam,DS}, the \emph{topological entropy of $G$} is $$\mathbf E_{top}(G)=\{h_{top}(\phi)\mid \phi\in\End(G)\}.$$

\medskip
Some of the most relevant problems concerning discrete dynamical systems deal with the values of entropy, for instance the problem of the existence of topological automorphisms of compact groups with arbitrary small topological entropy. This problem is certainly the most outstanding and can be written equivalently for continuous endomorphisms as follows: 
\begin{equation}\label{Lehmer}
\text{is}\  \inf\{\mathbf E_{top}(G)\setminus\{0\}\mid G\ \text{compact group}\}>0?
\end{equation}

It  is equivalent to the celebrated Lehmer's problem in number theory (see \cite{Sch}). In fact, 
$$\inf\{\mathbf E_{top}(G)\setminus\{0\}\mid G\ \text{compact group}\}=\inf(\{h_{top}(\phi)\mid \phi\in\Aut(\widehat\Q^n),\ n\in\N\}\setminus\{0\})$$ and, for $n\in\N$ and $\phi\in\Aut(\widehat\Q^n)$, the so-called Yuzvinski's formula states that $h_{top}(\phi)$ coincides with the Mahler measure of the characteristic polynomial of $\phi$ taken with integer coefficients (see \cite{LW,Y}). Lehmer's problem asks whether or not the infimum of all positive values of the Mahler measure is zero (see \cite{Lehmer}).

\smallskip
A positive answer to the question in \eqref{Lehmer} would imply that the set $$\mathbf E_{top}=\{\mathbf{E}_{top}(G)\mid G\ \text{compact group}\}$$
of all possible values of the topological entropy of continuous endomorphisms of compact groups is countable, while a negative answer would imply that $\mathbf E_{top}=\R_{\geq0}\cup\{\infty\}$ (see \cite{B,DG-islam,Sch,Y}, see also \cite{DG_Adv} for the algebraic counterpart).

\smallskip
In the larger class of locally compact groups the counterpart of \eqref{Lehmer} has an easy answer. In fact, $\mathbf E_{top}(\R)=\R_{\geq0}$; more precisely, every continuous endomorphism of $\R$ has finite topological entropy and for every non-negative real $r$ there exists a topological automorphism of $\R$ of topological entropy $r$ (see \cite{B,Walters}, see also Remark~\ref{Rrem} below). This is why this paper, following the direction of \cite{DS} in the compact case, studies the problem of the finiteness of the topological entropy for locally compact groups.  
In order to pursue this scope, we follow \cite{DS} and introduce
$$\mathfrak E_0=\{G\mid \mathbf E_{top}(G)=\{0\}\}\quad \text{and}\quad \mathfrak E_{<\infty}=\{G\mid \infty\not\in\mathbf E_{top}(G)\}.$$ 
We start recalling the results from \cite{DS} about locally compact abelian groups in $\mathfrak E_{<\infty}$ and in $\mathfrak E_0$.

\smallskip
Following \cite{DPS} and denoting by $\Prm$ the set of all primes, for $p\in\Prm$ we say that an element $x$ of a locally compact group $G$ is \emph{topologically $p$-torsion} if $x^{p^n}\to 1$ in $G$, and let $G_p=\{x\in G\mid x\ \text{topologically $p$-torsion}\}$ be the \emph{topological $p$-component} of $G$. We denote by $c(G)$ the connected component of $G$ and by $B(G)$ the largest compactly covered subgroup of $G$. According to \cite{HHR}, a \emph{locally compact $p$-group} is a locally compact group $G$ such that $G=G_p$.

\begin{theorem}[{See \cite[Theorems A, B and C and Corollary 2]{DS}}]\label{origin}
Let $G$ be a locally compact abelian group. 
\begin{enumerate}[(a)]
\item If $G \in \mathfrak{E}_{<\infty}$, then $\dim(G) < \infty$;  this implication can be inverted when $G$ is compact and $G/c(G)\in \mathfrak{E}_{<\infty}$ (in particular, when $G$ is compact and connected). 
\item If $G \in \mathfrak{E}_0$, then $G$ is totally disconnected; if $G$ is compact and totally disconnected, then $G \in \mathfrak{E}_0$ if and only if $G \in \mathfrak{E}_{<\infty}$.
\item In case $G$ is compact,  $G \in \mathfrak{E}_0$ if and only if $G$ is totally disconnected and $G_p \in \mathfrak{E}_{0}$ for every $p\in\Prm$. 
\end{enumerate}
\end{theorem}

Items (a) and (b) of the above theorem suggest to treat the case of totally disconnected locally compact abelian groups $G$ (it is worth noting that no complete reduction to the totally disconnected case is available -- see Question~\ref{ques0}(a)). We handle the case when $G$ is also compactly covered (i.e., each element of $G$ is contained in some compact subgroup of $G$); following the terminology from \cite{HHR}, we call those locally compact groups $G$ \emph{periodic}. 

\smallskip
 For periodic locally compact abelian groups we have the following reduction to locally compact abelian $p$-groups extending Theorem~\ref{origin}(c).

\begin{theorem}\label{red1}\label{prp:prod}
Let $G$ be a periodic locally compact abelian group. Then: 
\begin{enumerate}[(a)]
  \item $G \in \mathfrak E_0$ if and only if $G_p \in \mathfrak E_0$ for every $p\in\Prm$; 
  \item $G \in \mathfrak E_{<\infty}$ if and only if $G_p \in \mathfrak E_{<\infty}$  for every $p\in\Prm$ and $G_p \in \mathfrak E_0$ for almost all $p\in\Prm$. 
\end{enumerate}
\end{theorem}

From now on, with $p$ we always denote a prime number.
A locally compact abelian $p$-group $G$ is a $\Z_p$-module, so we can consider its $\Z_p$-rank $\rk_{\Z_p}(G)$ and its $p$-rank $r_p(G)$ as well.  When $\rk_{\Z_p}(G)$ and $r_p(G)$ are finite, let $$\mathrm{rank}_p(G)=\rk_{\Z_p}(G) + r_p(G).$$
According to \cite[Lemma 3.91]{HHR}, $\mathrm{rank}_p(G)$ is the smallest $n\in\N$ such that every topologically finitely generated subgroup of $G$ is generated by at most $n$ elements. For example the group $\mathbb{Q}_p$ of $p$-adic numbers is a torsion-free periodic locally compact abelian $p$-group with $r_p(G)=0$ while $\rank_p(\Q_p) = \rk_{\Z_p}(\Q_p)=1$.

Motivated by the following remark, we impose the additional restriction to have finite $\rank_p$ on the locally compact abelian $p$-groups.
Very roughly speaking, the finiteness of the cardinal invariant $\rank_p$ will replace (only to a certain extent -- see Remark~\ref{exotic:groups}(a)) the finiteness of the dimension in Theorem~\ref{origin}(a) as far as the description of the locally compact abelian $p$-groups in $\mathfrak E_{<\infty}$ is concerned. 

\begin{remark}\label{exotic:groups}
Here we recall some results from \cite{DS}, showing that the question of a complete classification of the compact abelian $p$-groups $G$ (i.e., abelian pro-$p$-groups) belonging to $\mathfrak{E}_{0}$ is far from a satisfactory understanding.
\begin{enumerate}[(a)]
\item By \cite[Theorem E]{DS} (its proof makes essential use of \cite[Theorem 13]{DGSZ}), there exists a family of $2^\mathfrak c$ many pairwise non-isomorphic compact abelian $p$-groups $K$ of weight $\mathfrak c$ in $\mathfrak E_0$ with $\rank_p(K)$ infinite.  Therefore, $\rank_p(G)< \infty$ is not a necessary condition neither for  $G\in \mathfrak E_{<\infty}$ nor (even) for $G\in\mathfrak E_0$.

\item There exists a class of compact abelian $p$-groups $K$ with $\mathbf E_{top}(K)=\{0,\infty\}$ (i.e., for every $\phi\in \End(K)$ one has $h_{top}(\phi)\in \{0,\infty\}$ and no intermediate values of the topological entropy can be attained). The groups $\Z_p^\N$ and $\prod_{n\in\N_+} \Z(p^n)$ have this property, and a complete description of this class can be found in \cite[Theorem D]{DS}. 
\end{enumerate}
\end{remark}

We recall the following recent characterization of the structure of locally compact abelian $p$-groups with finite $\mathrm{rank}_p$.

\begin{theorem}[{See \cite[Theorem 3.97]{HHR}}]\label{p-rank-finite} 
A locally compact abelian $p$-group $G$ has $\mathrm{rank}_p(G)<\infty$ if and only if it is isomorphic to
\begin{equation}\label{HHR}
\Z_p^{n_1} \times\Q_p^{n_2}\times \Z(p^\infty)^{n_3} \times F_p
\end{equation}
for some integers $n_1,n_2,n_3\in\N$ and a finite $p$-group $F_p$ with $r_p(F_p)=n_4\in\N$. In particular,
\[\rank_p(G)=n_1+n_2+n_3+ n_4,\] 
with $\rk_{\Z_p}(G)=n_1+n_2$ and $r_p(G)=n_3+n_4$.
\end{theorem}

This result provides a nice property of the cardinal invariant $\rank_p$, namely its preservation under Pontryagin duality. Indeed, if $G\cong\Z_p^{n_1} \times\Q_p^{n_2}\times \Z(p^\infty)^{n_3} \times F_p$, then $\widehat{G}\cong \Z_p^{n_3} \times\Q_p^{n_2}\times \Z(p^\infty)^{n_1} \times F_p$, and so $\rank_p(\widehat{G}) = \rank_p(G)$, while $\rk_{\Z_p}(\widehat{G})$ and $\rk_{\Z_p}(G)$, and similarly $r_p(\widehat{G})$ and $r_p(G)$, need not coincide.

\medskip
It turns out that all locally compact abelian $p$-groups with finite $\rank_p$ are in $\mathfrak E_{<\infty}$, and one can characterize those of them that are actually in $\mathfrak E_0$: 
\begin{theorem}\label{main1}
If $G$ is a locally compact abelian $p$-group with $\mathrm{rank}_p(G)<\infty$, then $G \in \mathfrak E_{<\infty}$. 
 Moreover, $G\in \mathfrak E_0$ if and only if $n_2=0$ in \eqref{HHR}.
\end{theorem}

The condition $n_2=0$ in Theorem~\ref{main1} means that the maximal divisible subgroup $d(G)$ of $G$ has $\rk_{\Z_p}(d(G))=0$, that is, $d(G)$ is discrete.

\smallskip
Theorem~\ref{main1} has to be compared with its counterpart for compact abelian groups; indeed, we have seen in Theorem~\ref{origin}(b) that when $G$ is a totally disconnected compact abelian group, $G \in \mathfrak E_{<\infty}$ precisely when $G \in \mathfrak E_0$.
More specifically, if $G$ is a compact abelian $p$-group  with $\rank_p(G)<\infty$, then $G\cong\Z_p^{n}\times F_p$ for some $n\in\N$ and a finite $p$-group $F_p$, by Theorem~\ref{p-rank-finite}. According to Proposition \ref{particularbutgeneral}, $\Z_p^{n}\times F_p\in\mathfrak E_0$ (this follows also from Theorem~\ref{main1}).

\smallskip
Making use of the ``local" Theorem~\ref{main1}, in Theorem~\ref{main2} we extend the characterization of the classes $\mathfrak E_{<\infty}$ and $\mathfrak E_0$ within a significantly larger class of locally compact abelian groups $G$.  Note that $$\varpi(G)=c(G)+B(G),$$ which is a fully invariant open subgroup of $G$ (see \cite[Proposition 3.3.6]{DPS}, see also \cite[pp. 29--30]{DG-BT}).

\begin{theorem}\label{main2}
Let $G$ be a locally compact abelian group with $c(B(G))=\{0\}$ and $\mathrm{rank}_p(G_p)<\infty$ for every $p\in\Prm$. Then:
\begin{enumerate}[(a)]
\item  $G \in \mathfrak E_{<\infty}$ holds whenever $\sum_{p\in \Prm} \rk_{\Z_p}(d(G_p))< \infty$;
\item $G\in\mathfrak E_{0}$ holds whenever $G$ is totally disconnected and $\rk_{\Z_p}(d(G_p)) = 0$ (i.e., $d(G_p)$ is discrete) for every $p\in\Prm$.
\end{enumerate}
If $G= \varpi(G)$, then also the converse implications hold in (a) and (b).
\end{theorem}

\medskip
The scale function $s(\phi)$ was introduced by Willis \cite{Willis} for continuous endomorphisms $\phi$ of totally disconnected locally compact groups (see Section~\ref{scalesec}).
Since by definition the scale is always finite, and since every locally compact abelian $p$-group $G$ with $\rank_p(G)<\infty$ is in $\mathfrak E_{<\infty}$ by Theorem~\ref{main1}, it makes sense to verify whether $h_{top}(\phi)=\log s(\phi)$ or not for every $\phi\in\End(G)$. Indeed we prove this equality, as reported below.

\begin{theorem}\label{main4}
Let $G$ be a locally compact abelian $p$-group with $\rank_p(G)<\infty$ and let $\phi\in\End(G)$. Then $h_{top}(\phi)=\log s(\phi)$.
\end{theorem}

The precise connection between the scale and the topological entropy was given in \cite{BDG,GBV} (see Section~\ref{scalesec}). In particular Theorem~\ref{main4} was already known for $G=\Q_p^n$ and $\phi$ a topological automorphism. 

\medskip
The proofs of Theorem~\ref{main1} and Theorem~\ref{main2} make substantial use of the so called Addition Theorem that we discuss in the sequel. 

We say that \emph{the Addition Theorem holds} for a triple $G,\phi,H$ of a topological group $G$, $\phi\in\End(G)$ and a $\phi$-invariant closed normal subgroup $H$ of $G$, and we briefly say that $AT(G,\phi,H)$ \emph{holds}, if  
\begin{equation}\label{AT}
h_{top}(\phi) =h_{top}(\phi\restriction_H)+h_{top}(\bar \phi_{G/H}),
\end{equation}
where $\bar\phi_{G/H}\in\End(G/H)$ is induced by $\phi$. 
The relevant formula \eqref{AT} can be briefly resumed by saying that in the following commutative diagram
\begin{equation*}
\begin{CD}
0 @>>>H@>\iota>>G@>\pi>>G/H@>>>0\\
@. @V\phi\restriction_HVV @V\phi VV @V\bar \phi VV\\
0 @>>>H@>\iota>>G@>\pi>>G/H@>>>0\\
\end{CD}
\end{equation*}
the topological entropy of the middle vertical arrow is the sum of the topological entropies of the remaining two vertical arrows. 

Similarly, we say that \emph{$AT(G)$ holds} if $AT(G,\phi,H)$ holds for every $\phi\in\End(G)$ and for every $\phi$-invariant closed normal subgroup $H$ of $G$.

\smallskip
 It is known that $AT(G,\phi,H)$ holds when $G$ is compact (see \cite{Dik+Manolo,Y}), and also when $G$ is totally disconnected and locally compact and $\phi$ is a topological automorphism of $G$ (see \cite{GBV}). However the validity of the Addition Theorem in the general case of locally compact groups and their continuous endomorphisms is not yet established even in the abelian setting, to the best of our knowledge. (A wrong proof of the Addition Theorem for locally compact abelian groups appeared in \cite{P} -- see \cite{DSV} for more details.)

\smallskip
The next theorem provides a formula reducing the computation of the topological entropy for totally disconnected locally compact abelian groups to the case of locally compact abelian $p$-groups. It allows for a convenient reduction of the Addition Theorem for totally disconnected locally compact abelian groups to the case of locally compact abelian $p$-groups.

\begin{theorem}\label{AT:AT}\label{Gp}
Let $G$ be a totally disconnected locally compact abelian group. Then, for every $\phi\in\End(G)$,
$$h_{top}(\phi)=\sum_{p\in \mathbb P} h_{top}(\phi\restriction_{G_{p}}).$$
\begin{enumerate}[(a)]
\item If $AT(G_p)$ holds for every for $p\in\Prm$, then also $AT(G)$ holds.
\item In case $G$ is periodic, $AT(G)$ holds if and only if $AT(G_p)$ holds for every $p\in\Prm$.
\end{enumerate}
\end{theorem}

Moreover, we prove the following version of the Addition Theorem, leaving open the problem in the general case of locally compact abelian $p$-groups of finite $\rank_p$ (see Question~\ref{ques1}).

\begin{theorem}\label{ATQp}
For $n\in\N$, $AT(\Q_p^n)$ holds.
\end{theorem}

\medskip
The paper is organized as follows.

In Section~\ref{sechtop} we recall some useful results on topological entropy; among them we mention the $p$-adic counterpart of the so-called Yuzvinski's formula (see Theorem~\ref{yuzvinskip}), which gives an explicit computation of the values of the topological entropy of continuous endomorphisms of $\Q_p^n$.

In Section~\ref{ATsec} we discuss some instances of the Addition Theorem, and in particular we prove Theorems~\ref{AT:AT} and~\ref{ATQp}. 

In Section~\ref{Mainsec} we prove Theorems~\ref{prp:prod},~\ref{main1} and~\ref{main2}. 
We make use of a direct computational rule for the topological entropy of a continuous endomorphism of a locally compact abelian $p$-group with finite $\mathrm{rank}_p$ (see Proposition~\ref{particularbutgeneral}), more precisely, one can reduce this computation to the mere use the $p$-adic Yuzvinski's formula.

In Section~\ref{scalesec} we recall the definition of the scale and the $p$-adic Yuzvinski's formula for the scale; we verify some basic properties in the abelian setting in order to prove Theorem~\ref{main4} (again as a consequence of Proposition~\ref{particularbutgeneral}).

Section~\ref{Hsec} contains some final comments and open questions, regarding those locally compact groups that have not been treated in the main body of the paper. In the first place, some attention is paid to the case of not necessarily totally disconnected locally compact abelian groups. 
The remaining part concerns the non-abelian setting, with a particular emphasis on the Heisenberg group on $\Z_p$ and $\Q_p$.

\subsection*{Notation and terminology}\label{NT}

As usual, $\R$ denotes the reals and $\R_{\geq0}=\{r\in\R\mid r\geq0\}$, $\Q$ the rationals, $\Z$ the integers, $\N$ the natural numbers and $\N_+$ the positive integers. We denote by $\Prm$ the set of all primes. 
For $p\in\Prm$, $\Z(p)$ is the finite cyclic group of order $p$, $\Z(p^\infty)$ the Pr\" ufer group, $\Z_p$ denotes the ring of $p$-adic integers, and $\Q_p$ the field of $p$-adic numbers.

For an abelian group $G$, we denote by $t(G)$ the torsion subgroup of $G$, by $D(G)$ the divisible hull of $G$, and by $d(G)$ the largest divisible subgroup of $G$. 

For $p\in\Prm$ and an abelian group $G$, we use $r_p(G)$ to denote the $p$-rank of $G$, that is $r_p(G) = \dim_{\Z/p\Z}G[p]$ where $G[p]=\{x\in G\mid px=0\}$; moreover, for a $\Z_p$-module $M$ we denote by $\mathrm{rk}_{\Z_p}$ its $\Z_p$-rank.

For a topological group $G$ we denote by $\End(G)$ the set of all continuous endomorphisms of $G$ and by $\Aut(G)$ the set of all topological automorphisms of $G$. For a locally compact abelian group $G$ we denote by $\widehat G$ its Pontryagin dual group.

For undefined symbols and terms, see \cite{DPS,HHR,HR,hofmor}.

\subsection*{Acknowledgements.} It is a pleasure to thank the referee for the helpful comments. The first and the second author thank GNSAGA of Indam, the third author thanks DMIF of Udine (Italy) for Grant No. PRID2017  and NRF of South Africa for Grant No. 118517.

\section{Some background on topological entropy}\label{sechtop}

Some useful facts on the topological entropy for locally compact groups are listed below. 

\begin{lemma}[{See \cite[Lemma 3.6(1)]{GBV}}]\label{monotonia}
Let $G$ be a locally compact group, $\phi\in\End(G)$ and $U,V\in\mathcal B(G)$. 
If $U\leq V$, then $H_{top}(\phi,V)\leq H_{top}(\phi,U)$. 
In particular, if $\mathcal B\subseteq \mathcal C(G)$ is a local base of $G$, then $h_{top}(\phi)=\sup\{H_{top}(\phi,U)\mid U\in\mathcal B\}$. 
\end{lemma}

\begin{corollary}[See \cite{DS}]\label{invbase} 
Let $G$ be a locally compact group and $\phi\in\End(G)$. If there exists a local base $\mathcal B\subseteq \mathcal C(G)$ of $G$ consisting of $\phi$-invariant subgroups, then $h_{top}(\phi)=0$. In particular, $\Z_p^n\in\mathfrak E_0$ for every $p\in\Prm$ and $n\in\N$.
\end{corollary}

Item (a) of the next lemma reveals the monotonicity of the topological entropy with respect to taking restrictions to invariant closed subgroups or quotients with respect to such subgroups. Item (b), which easily follows from (a), shows that $h_{top}$ is an invariant for topological dynamical systems $(G,\phi)$, where $G$ is a locally compact group and $\phi\in\End(G)$.

For $G,H$ locally compact abelian groups, we say that $\phi\in\End(G)$ and $\psi\in\End(H)$ are \emph{conjugated} if there exists a topological isomorphism $\alpha:G\to H$ such that $\psi=\alpha\phi\alpha^{-1}$.

\begin{lemma}[See \cite{DG-islam}]\label{conju}\label{mon}
Let $G$ be a locally compact group and $\phi\in\End(G)$. 
\begin{enumerate}[(a)]
\item If $H$ a $\phi$-invariant closed subgroup of $G$. Then $h_{top}(\phi\restriction_H)\leq h_{top}(\phi)$.
If $H$ is also normal, then $h_{top}(\bar\phi_{G/H})\leq h_{top}(\phi)$ where $\bar\phi_{G/H}:G/H\to G/H$ is the endomorphism induced by $\phi$.
\item If $H$ is a locally compact abelian group, and $\psi\in\End(H)$ is conjugated with $\phi$, then $h_{top}(\phi)=h_{top}(\psi)$.
\end{enumerate}
\end{lemma}

When $G$ is a totally disconnected locally compact group, the computation of the topological entropy of $\phi\in\End(G)$ can be simplified. Indeed, for these groups van Dantzig  \cite{vD} proved that the family
$$\B(G)=\{U\leq G\mid \text{$U$ compact and open}\}\subseteq\mathcal C(G)$$
is a local base of $G$.  As noticed in \cite[Proposition 4.5.3]{DG-islam}, the topological entropy of $\phi$ can be computed as
\begin{equation}\label{htoptdlc}
h_{top}(\phi)=\sup\{H_{top}(\phi,U)\mid U\in\B(G)\},\ \text{where}\ \ H_{top}(\phi,U)=\lim_{n\to\infty}\frac{\log [U:C_{n}(\phi,U)]}{n}
\end{equation}
(here $C_{n}(\phi,U)\in\B(G)$, and so the index $[U:C_n(\phi,U)]$ is finite since $C_{n}(\phi,U)$ is open in the compact subgroup $U$); moreover $H_{top}(\phi,U)\in\log\N_+,$ hence
\begin{equation}\label{htoplogN}
\mathbf E_{top}(G)\subseteq\log\N_+\cup\{\infty\}
\end{equation}

\begin{remark}\label{discrete=0} 
If $G$ is a discrete group, then $G\in\mathfrak E_0$.
In fact, $\B(G)=\{U\leq G\mid U\ \text{finite}\}$,  so if $\phi\in\End(G)$ and $U\in\B(G)$, then $[U:C_n(\phi,U)]\leq |U|$ for every $n\in\N_+$.
Hence, $H_{top}(\phi,U)=0$ and consequently $h_{top}(\phi)=0$. 
\end{remark}

Theorem \ref{yuzvinskip} below is the $p$-adic counterpart of the Yuzvinski's formula from \cite{Y} explicitly computing the topological entropy of all continuous endomorphisms of $\widehat{\Q}$. The $p$-adic Yuzvinski's formula was given in \cite{LW} for topological automorphisms, the general case can be obtained also from its counterpart for the algebraic entropy proved in \cite{GBV0} together with the so-called Bridge Theorem from \cite{DG-BT} (see also \cite{GBV1}).

If $K_p$ is a finite extension of $\Q_p$, we denote by $|-|_p$ the unique extension of the $p$-adic norm $|-|_p$ of $\Q_p$ to $K_p$.
In particular, if $K_p= \Q_p$ and $0\ne x \in \Q_p$ is written as $x = p^\ell u$, where $\ell \in \Z$ and $u\in U(\Z_p)$ (i.e., $|u|_p= 1$), then $|x|_p=(1/p)^{\ell} = p^{-\ell}$. 

\begin{theorem}[See \cite{LW}]\label{yuz-pp}\label{yuzvinskip} 
Let $n\in\N_+$ and $\phi\in\End(\Q_p^n)$. Then
$$h_{top}(\phi)=\sum_{|\lambda_i|_p>1}\log|\lambda_i|_p,$$
where $\{\lambda_i\mid i\in\{1,\ldots,n\}\}$ are the eigenvalues of $\phi$ in some finite extension $K_p$ of $\Q_p$.
\end{theorem}  

In particular, the above highly non-trivial theorem implies that $\Q_p^n\in\mathfrak E_{<\infty}$.

\begin{example}\label{Qpp}
Let $n\in\N_+$ and consider the multiplication by $1/p$, that is, $\phi:\Q_p^n\to \Q_p^n$, $x\mapsto\frac{1}{p}x$. 
A straightforward computation (or an easy application of Theorem~\ref{yuzvinskip}) gives $h_{top}(\phi)=n\log p>0$. 
Taking into account the above observation as well, one has
$$\Q_p^n \in\mathfrak E_{<\infty}\setminus \mathfrak E_0.$$
\end{example}

\begin{remark}\label{Rrem}
An explicit formula, similar to the $p$-adic Yuzvinski's formula, for the computation of the topological entropy of continuous endomorphisms of $\R^n$, for $n\in\N$, is known from \cite{B}. It implies that, for $n\in\N_+$, 
\begin{equation}\label{Rn}
\R^n\in\mathfrak E_{<\infty}\setminus\mathfrak E_0.
\end{equation}
\end{remark}

\section{Addition Theorem}\label{ATsec}

\subsection{Basic facts}

We start recalling the following weak version of the Addition Theorem.

\begin{lemma}[See \cite{AKM,DG-islam,GBV,V}]\label{wAT} 
Let $G_1$, $G_2$ be locally compact groups that are either compact or totally disconnected, or isomorphic to $\R^n$ for some $n\in\N$, and $\phi_1\in\End(G_1)$, $\phi_2\in\End(G_2)$. Consider $G_1\times G_2$ with the product topology and $\phi_1\times\phi_2\in\End(G_1\times G_2)$. Then $h_{top}(\phi_1\times\phi_2)=h_{top}(\phi_1)+h_{top}(\phi_2)$.
\end{lemma}

We now introduce other three levels of Addition Theorem. 

\begin{definition}\label{atdefinition}
Let $G$ be a  topological group.
\begin{enumerate}[(a)]
\item If $H$ is a fully invariant closed subgroup of $G$, $$AT(G,H)$$ means that $AT(G,\phi,H)$ holds for every $\phi\in\End(G)$.
\item If $\phi\in\End(G)$ and $H$ is a $\phi$-invariant closed normal subgroup of $G$, $$AT_0(G,\phi,H)\ (\text{respectively,}\ AT^0(G,\phi,H))$$ means that
$h_{top}(\phi)=h_{top}(\phi\restriction_H)$ and $h_{top}(\bar \phi_{G/H})=0$ (respectively, $h_{top}(\phi)=h_{top}(\bar \phi_{G/H})$ and $h_{top}(\phi\restriction_H)=0$).
\item  If $H$ is a fully invariant closed subgroup of $G$, $$AT_0(G,H)\ (\text{respectively,}\ AT^0(G,H))$$ means that 
$AT_0(G,\phi,H)$ (respectively, $AT^0(G,\phi,H)$) holds for every $\phi\in\End(G)$.
\end{enumerate}
\end{definition}

The following basic case of Addition Theorem was already proved in \cite[Corollary 4.17]{DSV} in the general case of all topological groups, here we give a short proof for the case of locally compact groups.

\begin{lemma}\label{DS4.17}
Let $G$ be a locally compact group, $\phi\in\End(G)$ and $H$ a $\phi$-invariant open normal subgroup of $G$. Then $AT_0(G,\phi,H)$ holds.
\end{lemma}
\begin{proof} 
By the hypotheses $\mathcal C(H)\subseteq \mathcal C(G)$ and $\mathcal C(H)$ is a local base of $G$, so $h_{top}(\phi)=h_{top}(\phi\restriction_H)$ by Lemma~\ref{monotonia}. Since $G/H$ is discrete, $h_{top}(\bar\phi_{G/H})=0$ by Remark~\ref{discrete=0}.
\end{proof}

\begin{corollary}\label{DS4.170} 
Let $G$ be a locally compact group and $H$ a fully invariant open subgroup of $G$. Then $AT_0(G,H)$ holds.
\end{corollary}

Next we give a useful application of Lemma~\ref{DS4.17}.

\begin{corollary}\label{new:lemma1}  
Let $G$ be a locally compact abelian group and endow $D(G)$ with the unique group topology that makes $G$ an open topological subgroup of $D(G)$. Then every $\phi\in\End(G)$ admits a continuous extension $\tilde \phi\in\End(D(G))$, and $AT_0(D(G),\tilde\phi,G)$ holds.
\end{corollary}
\begin{proof} 
The existence of such extension $\tilde\phi$ of $\phi\in\End(G)$ to $D(G)$ follows from the fact that $D(G)$ is divisible. The continuity of $\tilde \phi$ follows from that of $\phi$ since $G$ is open in $D(G)$. By Lemma~\ref{DS4.17}, $AT_0(D(G),\tilde\phi,G)$ holds.
\end{proof}    

The next result shows that in order to verify $AT(G)$ it suffices to verify $AT(N)$ for a fully invariant open subgroup of $G$.

\begin{proposition}\label{AT:wrt:open}
Let $G$ be a locally compact group and $N$ a fully invariant open subgroup of $G$.
If $AT(N)$ holds, then also $AT(G)$ holds.
\end{proposition}
\begin{proof} 
Let $\phi \in \End(G)$ and let $H$ be a $\phi$-invariant closed normal subgroup of $G$. 
Since $N$ is fully invariant, $\psi= \phi\restriction_N\in \End(N)$ is well defined. As $AT(N, \psi,H\cap N)$ holds by hypothesis, we have that
\begin{equation}\label{NewProp}
h_{top}(\psi) = h_{top}(\psi\restriction_{H\cap N}) + h_{top}(\bar \psi_{N/H\cap N}).
\end{equation}
By Lemma~\ref{DS4.17}, 
\begin{equation}\label{NewProp0}
h_{top}(\psi) = h_{top}(\phi)\ \mbox{ and } \ h_{top}(\psi\restriction_{H\cap N})= h_{top}(\phi\restriction_{H}).
\end{equation}
 It remains to prove that
\begin{equation}\label{NewProp1} 
h_{top}(\bar \psi_{N/H\cap N}) = h_{top}(\bar \phi_{G/H}).
\end{equation}
Indeed, \eqref{NewProp}, \eqref{NewProp0} and \eqref{NewProp1} give that $AT(G,\phi,H)$ holds.

The open subgroup $N+H/H$ of $G/H$ is $\bar \phi_{G/H}$-invariant,  so Lemma~\ref{DS4.17} gives that, letting $\xi$ be the restriction of $\bar \phi_{G/H}$ to $N+H/H$, $$h_{top}(\xi) = h_{top}(\bar \phi_{G/H}).$$
The quotient map $q:G\to G/N$ is open, so is its restriction $q\restriction_N:N \to q(N) = N+H/H$ onto the open subgroup $N+H/H$ of $G/H$.
Hence, the canonical continuous isomorphism $t: N/H\cap N \to N+H/H$ is a topological isomorphism, and $\xi$ is conjugated to $\bar \psi_{N/H\cap N}$ via $t$. Therefore, $$h_{top}(\bar \psi_{N/H\cap N}) = h_{top}(\xi).$$
This equality, along with the previous one, yields \eqref{NewProp1}.
\end{proof}

\subsection{Totally disconnected locally compact groups}

The following proposition plays a prominent role in the proof of our main results.
 
\begin{proposition}\label{fundamentalstep}
 Let $G_1$ be a totally disconnected locally compact abelian group, $G_2$ a discrete abelian group, and let $G = G_1\times G_2$ equipped with the product topology. If $\phi\in\End(G)$ and $G_2$ is $\phi$-invariant, then  $AT^0(G,\phi,G_2)$ holds.
\end{proposition}
\begin{proof} 
There exist $\phi_1\in\End(G_1)$, $\phi_2\in\End(G_2)$ and a continuous homomorphism $\phi_3: G_1 \to G_2$, such that for every $(x,y) \in G$ $$\phi(x,y) = (\phi_1(x), \phi_2(y) + \phi_3(x)).$$  
Let $K\in\B(G_1)$. Since $G_2$ is discrete, the compact subgroup $\phi_3(K)$ of $G_2$ must be finite.  
Therefore, $V= \ker\phi_3 \cap K$ has finite index in $K$, so $V\in \B(G_1)$, that is, $V \times \{0\}\in\B(G)$.

For every $U\in \B(G)$ contained in $V$,  and for every $n\in\N_+$,
\begin{align*}
C_{n}(\phi,U\times \{0\})&= (U\times\{0\})\cap \phi^{-1}(U\times\{0\}) \cap\ldots\cap \phi^{-n+1}(U\times\{0\})\\
&=\{(x,0)\in U\times\{0\}\mid \phi^k_1(x) \in U\ \forall k\in\{0,\ldots,n-1\}\}\\
&=\{(x,0)\in U\times\{0\}\mid x \in\phi^{-k}(U)\ \forall k\in\{0,\ldots,n-1\}\}\\
&=C_{n}(\phi_1,U)\times \{0\}.
\end{align*}
This yields 
$H_{top}(\phi,U\times\{0\}) = H_{top}(\phi_1,U)$.
Since $\mathcal B=\{U\times \{0\}\mid U\in \B(G_1),\ U\subseteq V\}\subseteq \B(G)$ is a local base of $G$, by Lemma~\ref{monotonia} we conclude that
$h_{top}(\phi)=h_{top}(\phi_1)$.

Let $\alpha:G/G_2\to G_1$, $(x,y)+G_2\mapsto x$, be the canonical isomorphism. Then  $\bar \phi: G/G_2 \to G/G_2$ is conjugated to $\phi_1$ by $\alpha$, and so we can conclude that $h_{top}(\phi) = h_{top}(\bar \phi)$ by Lemma~\ref{conju}(b).  On the other hand, $h_{top}(\phi\restriction_{G_2})=h_{top}(\phi_2) = 0$, as $G_2$ is discrete. Therefore, $AT^0(G,\phi,G_2)$ holds true. 
\end{proof} 

The following is a direct consequence of Proposition~\ref{fundamentalstep} when $G_2$ is fully invariant whenever identified with the subgroup $\{0\}\times G_2$ of $G$.

\begin{corollary}\label{fundamentalstepcor} 
Let $G_1$ be a totally disconnected locally compact abelian group, $G_2$ a discrete abelian group, and let $G = G_1\times G_2$ equipped with the product topology. If $G_2$ is fully invariant, then $AT^0(G,G_2)$ holds.
\end{corollary}

The above corollary applies directly in the following one.

\begin{corollary}\label{cor:fundamentalstep}
Let $G_1$ be a totally disconnected locally compact abelian group and let $G_2$ be a discrete abelian group.
Suppose that one of the following conditions is fulfilled: 
\begin{enumerate}[(a)]
\item $G_1$ is torsion-free and $G_2$ is torsion; 
\item $G_1$ is reduced and $G_2$ is divisible. 
\end{enumerate}
Then, for $G = G_1\times G_2$ equipped with the product topology, $AT^0(G,G_2)$ holds.
\end{corollary}

We see now that under some suitable conditions the validity of $AT(G)$ implies the validity of $AT(G_1)$ for a topological direct summand $G_1$ of $G$.

\begin{proposition}\label{neeew}
Let $G=G_1\times G_2$ be a locally compact abelian group, where $G_1$ and $G_2$ are either compact, or totally disconnected, or isomorphic to $\R^n$ for some $n\in\N$.
If $AT(G)$ holds then both $AT(G_1)$ and $AT(G_2)$ holds.
\end{proposition}
\begin{proof}
Obviously, it is enough to check $AT(G_1)$. To this end fix $\phi_1\in \End(G_1)$ and a $\phi_1$-invariant closed subgroup $H$ of $G_1$. 

Consider $\phi = \phi_1 \times 0 \in \End(G)$, so that $\phi\restriction_{G_1} = \phi_1$ and $\phi\restriction_{G_2}= 0$. 
 This yields $h_{top}(\phi\restriction_{G_2})= 0$. Therefore, Lemma~\ref{wAT} gives 
\begin{equation}\label{loc:prod0}
h_{top}(\phi) = h_{top}(\phi\restriction_{G_1})= h_{top}(\phi_1). 
\end{equation} 
Note that  $H$ is also a $\phi$-invariant closed subgroup of $G$.  Since $AT(G,\phi,H)$ holds in view of our hypothesis $AT(G)$, we have that
\begin{equation}\label{loc:prod}
h_{top}(\phi) = h_{top}(\phi\restriction_H)+ h_{top}(\bar\phi_{G/H}) = h_{top}(\phi_1\restriction_H)+ h_{top}(\bar\phi_{G/H}).
\end{equation}
Since $G/H\cong G_1/H \times G_2$ and $\bar\phi_{G/H}$ is conjugated to $\psi= \overline{(\phi_1)}_{G_1/H}\times \phi\restriction_{G_2}$,  by Lemma~\ref{wAT} and Lemma~\ref{conju}(b),
$$h_{top}(\bar\phi_{G/H}) = h_{top}(\psi)=h_{top}(\overline{(\phi_1)}_{G_1/H}).$$
Along with \eqref{loc:prod0} and \eqref{loc:prod}, this implies $h_{top}(\phi_1)=  h_{top}(\phi_1\!\! \restriction_H)+ h_{top}(\overline{(\phi_1)}_{G_1/H})$, i.e., $AT(G_1,\phi_1,H_1)$ holds. 
\end{proof}

Next we prove a natural instance of the Addition Theorem with respect to the kernel of the endomorphism.

\begin{theorem}\label{LAAAAAst:thm}
Let $G$ be a totally disconnected locally compact group and $\phi\in\End(G)$.
If there exists a local base $\B\subseteq\B(G)$ of $G$ such that, for every $U\in\B$, $\phi^{-n}(U)$ is normal in $G$ for every $n\in\N_+$, then $AT^0(G,\phi,\ker\phi)$ holds.
\end{theorem}
\begin{proof}
Clearly, $h_{top}(\phi\restriction_{\ker\phi})=0$, since $\phi\restriction_{\ker\phi}=0$. 

It remains to verify that $h_{top}(\phi)=h_{top}(\bar\phi)$, where we denote $\bar\phi=\bar\phi_{G/\ker\phi}$. Since $h_{top}(\bar\phi)\leq h_{top}(\phi)$ by Lemma~\ref{mon}(a), we are left with the converse inequality.

Denote by $\pi:G\to G/\ker\phi$ the canonical projection. Moreover, recall that $\pi(\B(G))=\{\pi(U)\mid U\in\B(G)\}\subseteq \B(G/\ker\phi)$ is a local base of $G/\ker\phi$ (see \cite[Corollary 2.5]{GBV}).
Let $U\in\B$ and $n\in\N_+$. It follows from our hypothesis that $N=\phi^{-1}(U)\cap\ldots\cap\phi^{-n+1}(U)$ is a normal subgroup of $G$, so $\pi(N)$ is a normal subgroup 
of $G/\ker \phi$. Therefore,
$$\frac{U}{C_n(\phi,U)}\cong \frac{U N}{N}\cong \frac{\pi(UN)}{\pi(N)}=\frac{\pi(U)\pi(N)}{\pi(N)}\cong\frac{\pi(U)}{\pi(U)\cap\pi(N)}=\frac{\pi(U)}{C_n(\bar\phi,\pi(U))};$$
hence, $$[U:C_n(\phi,U)]=[\pi(U):C_n(\bar\phi,\pi(U))].$$
Therefore, $H_{top}(\phi,U)= H_{top}(\bar\phi,\pi(U))$ and finally $h_{top}(\phi)=h_{top}(\bar\phi)$ by Lemma~\ref{monotonia} and \eqref{htoptdlc}.
\end{proof}

\subsection{Reduction to locally compact abelian $p$-groups}

Following \cite{DPS}, for a topological abelian group $G$ we denote by $B(G)$ the set of all compact elements of $G$, that are the elements of $G$ contained in a compact subgroup of $G$. Hence (by definition), $B(G)$ is the largest compactly covered subgroup of $G$, therefore it is fully invariant $G$. Also the subgroup $$\varpi(G)=B(G) + c(G)$$ is fully invariant, and, if $G$ is locally compact, $\varpi(G)$ is also open in $G$ (see \cite[Proposition 3.3.6]{DPS}). This means that $B(G)$ itself is open in $G$ if and only if $G$ contains no copies of $\R$, i.e., $c(G)$ is compact.

\begin{proposition}\label{red2}
Let $G$ be a locally compact abelian group. Then $AT_0(G,\varpi(G))$ holds.
\end{proposition}
\begin{proof} 
As mentioned above, the subgroup $\varpi(G)$ is open and fully invariant. Therefore, Corollary~\ref{DS4.170} applies.  
\end{proof}

If $G$ is a locally compact abelian group, for every $p\in\Prm$, the subgroup $G_p$ is compactly covered.
Since the closure $\overline{G_p}$ contains $c(B(G))=c(G)\cap B(G)$, and since $c(B(G))_p \ne \{0\}$ whenever $c(B(G))\ne \{0\}$,  $G_p$ is closed (equivalently, locally compact) if and only if $c(B(G))=\{0\}$ (i.e., $G$ contains no non-trivial compact connected subgroups). In particular, 
$G_p$ is closed precisely  when $G$ is totally disconnected.

\smallskip
Every periodic locally compact abelian group $G$ (i.e., $G=B(G)$ and $c(G)=\{0\}$) is a local product $\overset{loc}{\prod}_{p\in\Prm} (G_p,K_p)$, where $K\in\B(G)$ (see \cite{Braconnier}). Let us recall that this is the subgroup $$\left\{(x_p)_p\in\prod_{p\in\Prm}G_p\mid x_p\in K_p\ \text{for all but finitely many}\ p\in\Prm\right\}$$ of $\prod_{p\in\Prm}G_p$ endowed with the topology with respect to which $\prod_{p\in\Prm}K_p$ is open.

\begin{proof}[\bf Proof of Theorem~\ref{AT:AT}]
First we prove that, for the totally disconnected locally compact abelian group $G$,
\begin{equation}\label{Gpeq}
h_{top}(\phi)=\sum_{p\in \mathbb P} h_{top}(\phi\restriction_{G_{p}}).
\end{equation}
Since $c(G) =\{0\}$, $AT_0(G,B(G))$ holds, according to Proposition~\ref{red2}. 
Therefore, $h_{top}(\phi) = h_{top}(\phi\restriction_{B(G)})$. Since $G_p= B(G)_p$ for every $p\in\Prm$, we may assume without loss of generality that 
$G$ is periodic, i.e., $G= B(G)$. Hence, we can write $G$ as $G=\overset{loc}{\prod}_{p\in\Prm} (G_p,K_p)$, where $K_p\in\B(G_p)$ for every $p\in\Prm$ and $K=\prod_{p\in\Prm}K_p$. For every $p\in\Prm$, $G_p$ is fully invariant in $G$, so it makes sense to consider $\phi\restriction_{G_p}$. 

For each $n\in \N_+$ let $\mathbb P_n$ be the set of the first $n$ primes and consider the splitting 
$$G = G_{(n)}  \times H_n,\quad \text{where}\quad G_{(n)} = \prod_{p\in \mathbb P_n} G_p\quad\text{and}\quad H_n=\overset{loc}{\prod}_{p\in \mathbb P\setminus \mathbb P_n}(G_p,K_p).$$
Since both $G_{(n)} $ and $H_n$ are fully invariant in $G$, by Lemma~\ref{wAT} we have that
$$h_{top}(\phi\restriction_{B(G)}) = h_{top}(\phi\restriction_{G_{(n)}}) + h_{top}(\phi\restriction_{H_{n}})=
\sum_{p\in \mathbb P_n} h_{top}(\phi\restriction_{G_{p}})+h_{top}(\phi\restriction_{H_{n}}).$$
Since $h_{top}(\phi\restriction_{G_{p}})\geq 0$ for every $p\in\Prm$, this equality gives us \eqref{Gpeq} in both cases, when the series $\sum_{p\in \mathbb P} h_{top}(\phi\restriction_{G_{p}})$ converges and when it diverges. 

\medskip
We prove now the second part of the theorem.
\begin{enumerate}[(a)]
\item If $AT(G_p)$ holds for every for $p\in\Prm$, then also $AT(G)$ holds.
\item In case $G$ is periodic, $AT(G)$ holds if and only if $AT(G_p)$ holds for every $p\in\Prm$.
\end{enumerate}
First we assume that $G=B(G)$ is periodic and write $G=\overset{loc}{\prod}_{p\in\Prm} (G_p,K_p)$, where $K_p\in\B(G_p)$ for every $p\in\Prm$ and $ K = \prod_{p\in\Prm} K_p$. 

Assume that $AT(G)$ holds, fix $q\in\Prm$ and consider the group $G^* = \overset{loc}{\prod}_{p\in\Prm\setminus \{q\}} (G_p,K_p)$. Hence, $G = G^* \times G_q$. Therefore, $AT(G_p)$ follows from Proposition~\ref{neeew}. This concludes the proof of the additional implication in item (b).

Now assume that $AT(G_p)$ holds for every $p\in\Prm$. Let $\phi \in \End(G)$ and let $H$ be a $\phi$-invariant closed subgroup of $G$. Then $$H = \overset{loc}{\prod}_{p\in\Prm} (H_p,H_p \cap K_p),$$ where $H_p =H\cap G_p$. 
Indeed, $H$ is a periodic locally compact abelian group and for every $p\in\Prm$ the subgroup $H_p \cap K_p = H_p \cap K\in\B(H_p)$.
Moreover, $$G/H \cong \overset{loc}{\prod}_{p\in\Prm} (G_p/H_p, (K_p+H_p)/H_p),$$
where $(K_p+H_p)/H_p\cong H_p/(H_p\cap K_p)$.

By \eqref{Gpeq} applied to $G, \phi$, $G/H, \bar \phi$ and  $H, \phi\restriction_H$, we get 
 $$h_{top}(\phi)=\sum_{p\in \mathbb P} h_{top}(\phi\restriction_{G_{p}}), \ \ h_{top}(\phi_{G/H})=\sum_{p\in \mathbb P} h_{top}(\phi\restriction_{G_{p}/H_p}) \ \ \mbox{ and }\ \  h_{top}(\phi\restriction_H)=\sum_{p\in \mathbb P} h_{top}(\phi\restriction_{H_{p}}).$$

Since $AT(G_p)$ holds for every $p\in\Prm$, we have $$h_{top}(\phi\restriction_{G_p}) = h_{top}(\phi\restriction_{H_p}) + h_{top}(\phi\restriction_{G_p/H_p})$$ and we are done.

We end with the general case of the proof of (a), for a totally disconnected group $G$, supposing that $AT(G_p)$ holds for every $p\in\Prm$. 
Since $ c(G) =\{0\}$, $B(G)$ is a fully invariant open subgroup of $G$. Moreover, since $B(G)$ is periodic and $B(G)_p = G_p$ for every $p\in\Prm$, we deduce from the above argument that $AT(B(G))$ holds. By Proposition~\ref{AT:wrt:open} we conclude that $AT(G)$ holds as well. 
\end{proof}

Note that in the above proof the subgroups $G_{(n)} \times \prod_{p\in \mathbb P\setminus \mathbb P_n} K_p$ form an increasing chain of open subgroups of $G=\bigcup_{n\in\N} G_{(n)}$, but these subgroups need not be invariant.

\subsection{The Addition Theorem for $\Q_p^n$}

An important instance of Addition Theorem is given by the following example, where we have a locally compact abelian $p$-group which is not compact. 
A more general result will be given in Proposition~\ref{particularbutgeneral}, nevertheless we anticipate this particular case which can be checked directly, without any recourse to the highly non-trivial Theorem~\ref{yuzvinskip}. 

\begin{example}\label{roadforthevictory}
Let $G=\Z_p\times \Q_p$, for a prime $p$, and $\phi\in\End(G)$. Since $\Q_p=d(G)$, it is fully invariant, so in particular it is $\phi$-invariant.  We prove that 
\begin{equation}\label{ecco}
AT_0(\Z_p\times\Q_p,\Q_p)\ \text{holds}.
\end{equation}

Since $d(G)=\{0\}\times\Q_p$ is a fully invariant subgroup of $G$, there exist $\phi_1\in\End(\Z_p)$, $\phi_2=\phi\restriction_{\Q_p}\in\End(\Q_p)$, and a continuous endomorphism $\phi_3:\Z_p\to\Q_p$ such that for every $(x,y)\in G$, 
$$\phi(x,y)=(\phi_1(x),\phi_2(y)+\phi_3(x)).$$
Therefore, there exist $\xi_1\in\Z_p$ and $\xi_2,\xi_3\in\Q_p$ such that $\phi_i(z)=\xi_i z$ for $i\in\{1,2,3\}$.
Let $p^{\ell_i}=|\xi_i|_p$ with the convention that $\ell_i=-\infty$ in case $\xi_i=0$ since $|0|_p=0$. 

\medskip
Let us see that 
\begin{equation}\label{phi2}
h_{top}(\phi_2)=\begin{cases}0&\text{if}\ \ell_2\leq0,\\ \ell_2\log p & \text{if}\  \ell_2>0.\end{cases}
\end{equation}
In fact, consider the local base $\mathcal B=\{p^k\Z_p\mid k\in\N\}\subseteq\mathcal B(\Q_p).$
 Then, for every $k,n\in\N$, $$C_{n+1}(\phi_2,p^k\Z_p)=\begin{cases}p^k\Z_p &\text{if}\ \ell_2\leq 0,\\p^{k+n\ell_2}\Z_p& \text{if}\ \ell_2>0,\end{cases}$$ and so $$H_{top}(\phi_2,p^k\Z_p)=\begin{cases}0& \text{if}\ \ell_2\leq0,\\ \ell_2\log p&\text{if}\ \ell_2>0.\end{cases}$$
By Lemma~\ref{monotonia}, we obtain \eqref{phi2}.

Next we verify that 
\begin{equation}\label{phi}
h_{top}(\phi)=\begin{cases}0&\text{if}\ \ell_2\leq0\\ \ell_2\log p & \text{if}\  \ell_2>0.\end{cases}
\end{equation}
Consider the local base 
$\mathcal B=\{p^k\Z_p\times p^k\Z_p\mid k\in\N,\ k\geq\ell_2\}\subseteq\mathcal B(G)$.
For $U=p^k\Z_p\times p^k\Z_p\in\mathcal B$ and for every $n\in\N$, one has
$$C_{n+1}(\phi,U)=\begin{cases}U &\text{if}\ \ell_2\leq 0,\\ p^{k}\Z_p\times p^{k+n\ell_2}\Z_p &\text{if}\ \ell_2>0,\end{cases}$$
and so
$$[U:C_{n+1}(\phi,U)]=\begin{cases}1&\text{if}\ \ell_2\leq0, \\ p^{n\ell_2}&\text{if}\ \ell_2>0.\end{cases}$$
Hence, $$H_{top}(\phi,U)=\begin{cases}0&\text{if}\ \ell_2\leq0,\\ \ell_2\log p&\text{if}\ \ell_2>0.\end{cases}$$
By Lemma~\ref{monotonia}, we conclude that \eqref{phi} holds.

Now \eqref{phi2} and \eqref{phi} give \eqref{ecco}.

\smallskip
Consider $\bar\phi\in\End(G/\Q_p)$. Since $G/\Q_p\cong\Z_p$, by Corollary~\ref{invbase} we conclude $h_{top}(\bar\phi)=0$.
Hence, $AT_0(G,\phi,\Q_p)$ holds, as announced above.
\end{example}

\begin{proof}[\bf Proof of Theorem~\ref{ATQp}] 
Let $G=\Q_p^n$ and $\phi\in\End(\Q_p^n)$, and let $H$ be a $\phi$-invariant closed subgroup of $G$. Being a closed subgroup of $G$, $$H=\Z_p^m\times \Q_p^k\ \text{with}\ l=m+k\leq n.$$
Since $G$ is divisible, we may assume without loss of generality that $D(H)\subseteq G$. Moreover, $D(H)$ is a $\phi$-invariant closed subgroup of $G$. Indeed, being a closed subgroup, $H$ is a $\Z_p$-submodule of $\Q_p^n$, so its divisible hull $D(H)$ is a $\Q_p$-submodule of $G$;  the isomorphism $D(H)\cong\Q_p^l$ implies that $D(H)$ is locally compact, hence complete, and therefore $D(H)$ is a closed subgroup of $G$. 

Since $H$ is open in $D(H)$, by Corollary~\ref{new:lemma1} we have that 
\begin{equation}\label{ATeq1}
h_{top}(\phi\restriction_H)=h_{top}(\phi\restriction_{D(H)}).
\end{equation}
Since $D(H)/H$ is a discrete $\bar\phi_{G/H}$-invariant subgroup of $G/H$, and since the endomorphism induced by $\bar\phi_{G/H}$ on $(G/H)/(D(H)/H)$ is conjugated to $\bar\phi_{G/D(H)}$, by Lemma~\ref{conju}(b9 and by Proposition~\ref{fundamentalstep}, we have that
\begin{equation}\label{ATeq2}
h_{top}(\bar\phi_{G/H})=h_{top}(\bar\phi_{G/D(H)}).
\end{equation}

Therefore, it remains to show that $AT(G,\phi,D(H))$ holds, that is,
$$h_{top}(\phi)=h_{top}(\phi\restriction_{D(H)})+h_{top}(\bar\phi_{G/D(H)}).$$
Since $D(H)\cong\Q_p^l$ is divisible, then $G=D(H)\times L$, where $L\leq \Q_p^n$ and $L\cong\Q_p^{n-l}$.
Being $G $ a $\Q_p$-vector space, $\phi$ is a $\Q_p$-linear transformation, and since $D(H)$ and $L$ are $\Q_p$-linear subspaces of $\Q_p^n$ with $D(H)$ $\phi$-invariant, we have that $\phi$ is induced by a matrix $A\in M_n(\Q_p)$ of the form
$$A=\begin{pmatrix}
A_1 & 0\\
B & A_2
\end{pmatrix}$$
where $A_1\in M_l(\Q_p)$ induces $\phi\restriction_{D(H)}$ and $A_2\in M_{n-l}(\Q_p)$ induces $\phi_2:L\to L$, where $\phi_2$ is conjugated to $\bar\phi_{G/D(H)}$.
Since $A$ has the same eigenvalues of
$$\begin{pmatrix}
A_1 & 0\\
0 & A_2
\end{pmatrix},$$
which is the matrix of $\phi\restriction_{D(H)}\times \phi_2$,
by Theorem~\ref{yuzvinskip} and Lemma~\ref{conju}(b), we conclude that $$h_{top}(\phi)=h_{top}(\phi\restriction_{D(H)})+h_{top}(\phi_2)=h_{top}(\phi\restriction_{D(H)})+h_{top}(\bar\phi_{G/D(H)}).$$
This concludes the proof in view of \eqref{ATeq1} and \eqref{ATeq2}.
\end{proof}

We are not aware whether $AT(G)$ holds for every locally compact abelian $p$-group $G$ with $\mathrm{rank}_p(G)<\infty$
(see Question~\ref{ques1}).

\section{Locally compact abelian groups in $\mathfrak E_{<\infty}$ and $\mathfrak E_0$}\label{Mainsec}

\subsection{General facts on $\mathfrak E_{<\infty}$ and $\mathfrak E_0$}

The classes $\mathfrak E_0$ and $\mathfrak E_{<\infty}$ are obviously stable under taking direct summands. 
Moreover, they are also stable under taking extensions with respect to fully invariant closed subgroups satisfying the Addition Theorem in the sense of Definition~\ref{atdefinition}(b): 

\begin{lemma}\label{Lemma3space} 
Let $G$ be a locally compact abelian group with a fully invariant closed subgroup $H$ such that $AT(G,H)$ holds. 
\begin{enumerate}[(a)]
\item If $H\in \mathfrak E_{<\infty}$ and $G/H\in \mathfrak E_{<\infty}$, then $G \in \mathfrak E_{<\infty}$.
\item If $H\in \mathfrak E_0$ and $G/H\in \mathfrak E_0$, then $G \in \mathfrak E_0$.
\end{enumerate}
Assume that $AT_0(G,H)$ holds.
\begin{enumerate}[(a$'$)]
\item If $H\in \mathfrak E_{<\infty}$, then $G \in \mathfrak E_{<\infty}$.
\item If $H\in \mathfrak E_0$, then $G \in \mathfrak E_0$. 
\end{enumerate}
\end{lemma} 

Note that for $G \in \mathfrak E_0$ the conjunction of $H\in \mathfrak E_0$ and $G/H\in \mathfrak E_0$ for a fully invariant closed subgroup $H$ obviously implies that $AT(G,H)$ holds.

\begin{remark}
One may ask whether the implications in Lemma~\ref{Lemma3space} can be inverted. To show that the answer is negative, at least in the cases (a) and (b), we make use of the examples of compact abelian $p$-groups $G\in \mathfrak E_0$ of weight $\mathfrak c$ from Remark~\ref{exotic:groups}.
Let us see that $G$ has a  fully invariant closed subgroup $H$ with $G/H \not \in \mathfrak E_{<\infty}$. 
Since $|\widehat G| = w(G)=\mathfrak c$, one has $r_p(\widehat G)= |\widehat G| = \mathfrak c$, and consequently $\widehat G[p] \cong \bigoplus_\mathfrak c \Z(p)$. Therefore, $G/pG \cong \Z(p)^\mathfrak c \not \in \mathfrak E_{<\infty}$ (see Example~\ref{FGR}).

\end{remark}

The following is a direct application of Lemma~\ref{new:lemma1}.

\begin{lemma}\label{application}
Let $G$ be a locally compact abelian group and endow $D(G)$ with the unique group topology that makes $G$ an open topological subgroup of $D(G)$. 
\begin{enumerate}[(a)]
\item If $D(G) \in \mathfrak{E}_{<\infty}$, then $G \in \mathfrak{E}_{<\infty}$. 
\item If $D(G) \in \mathfrak{E}_0$, then $G \in \mathfrak{E}_0$.
\end{enumerate}
\end{lemma}

We are not aware if the implication in the conclusion of the Lemma~\ref{application} can be inverted, see Question~\ref{notaware}.

\subsection{Reduction to locally compact abelian $p$-groups}\label{proof:Main2}

Since we want to determine when a locally compact abelian group $G$ is in $\mathfrak E_{<\infty}$, the following lemma gives a sufficient condition in terms of the open fully invariant subgroup $\varpi(G)$ of $G$.

\begin{lemma}\label{last?}
Let $G$ be a locally compact abelian group.  If $\varpi(G) \in \mathfrak E_{<\infty}$, then $G \in \mathfrak E_{<\infty}$.
\end{lemma}
\begin{proof}
Follows from Lemma~\ref{Lemma3space}, as $AT_0(G,\varpi(G))$ holds by Proposition~\ref{red2}. 
\end{proof}

A locally compact abelian group $G$ can be identified with $\R^n\times G_0$ with the product topology, for some $n\in\N$ and with $\mathcal B(G_0)\neq\emptyset$. Since $\R^n$ contains no non-trivial subgroups, $B(G)\subseteq G_0$. On the other hand, $c(G)=\R^n\times K$, where $K$ is compact and connected,  so $\varpi(G)=B(G)+\R^n.$

While $B(G)$ is fully invariant in $G$, $\R^n$ is fully invariant in $G$ if and only if $B(G)$ is totally disconnected, that is, $c(B(G))=c(G)\cap B(G)=\{0\}$.
Under this assumption, $c(G)=\R^n$ and $\varpi(G)=B(G)\times\R^n$ topologically, and we see that $\varpi(G)\in\mathfrak E_{<\infty}$ when $B(G)\in \mathfrak E_{<\infty}$:

\begin{lemma}\label{Laaaast:lemma}
Let $G$ be a locally compact abelian group such that $c(B(G)) = \{0\}$. Then   
\begin{equation}\label{lasteq}
h_{top}(\phi)=h_{top}(\phi\restriction_{c(G)})+h_{top}(\phi\restriction_{B(G)})\ \text{for every}\ \phi\in\End(G).
\end{equation}
Consequently, if $B(G)\in \mathfrak E_{<\infty}$ then $G \in \mathfrak E_{<\infty}$. 
\end{lemma}
\begin{proof}
First we note that $AT_0(G,\varpi(G))$ holds by Proposition~\ref{red2}, hence $h_{top}(\phi) = h_{top}(\phi\restriction_{\varpi(G)})$.
Moreover, the open subgroup $G_1 = \varpi(G)$ of $G$ satisfies $B(G_1) = B(G)$ and $c(G_1) = c(G)$, so we may assume without loss of generality that $G = \varpi(G)$. 

As $c(G) = \R^n \times c(B(G))$ for some $n\in\N$, our hypothesis implies that $c(G) =\R^n$ and $c(G)\cap B(G)=c(B(G))=\{0\}$.
This gives the isomorphism $G \cong B(G) \times c(G)$. Since both $c(G)$ and $B(G)$ are fully invariant in $G$,  \eqref{lasteq} holds, 
by Lemma~\ref{wAT}.

The last assertion follows from \eqref{lasteq} and \eqref{Rn}.
\end{proof}

The following result covers Theorem~\ref{red1}. It permits the crucial reduction to the case of locally compact abelian $p$-groups.

\begin{proposition}\label{proofMain2}
Let $G$ be a locally compact abelian group such that $c(B(G)) = \{0\}$. Then: 
\begin{enumerate}[(a)] 
\item $G \in \mathfrak E_{<\infty}$ whenever $G_p\in \mathfrak E_{<\infty}$  for every $p\in\Prm$ and $G_p \in \mathfrak E_0$ for almost all $p\in\Prm$;
\item $G \in \mathfrak E_0$ whenever $G$ is totally disconnected and $G_p \in \mathfrak E_0$ for every $p\in\Prm$.
\end{enumerate}
 If $G = \varpi(G)$, then also the converse implications hold. In particular, $G \in \mathfrak E_{<\infty}$ if and only if $B(G) \in \mathfrak E_{<\infty}$.
\end{proposition}
\begin{proof}
If $\phi\in\End(G)$, by Theorem~\ref{Gp} and Lemma~\ref{Laaaast:lemma} we have $h_{top}(\phi)=h_{top}(\phi\restriction_{c(G)})+\sum_{p\in \mathbb P} h_{top}(\phi\restriction_{G_{p}})$. The conclusion follows since $h_{top}(\phi\restriction_{G_p})\in\log\N_+\cup\{\infty\}$ for every $p\in\Prm$ by \eqref{htoplogN}. 

Assuming $G = \varpi(G)$, we have a direct product $G = c(G) \times B(G)$, where $c(G)\cong\R^n$ for some $n\in\N$. Since for every $p\in\Prm$ the fully invariant closed subgroup $G_p=B(G)_p$ is a direct summand of $B(G)$, we have that $G_p$ is a direct summand of $G$. Therefore, $G \in \mathfrak E_0$ implies that $c(G)=0$ in view of \eqref{Rn}, and that each $G_p \in \mathfrak E_0$ by Lemma~\ref{wAT}. Similarly, $G \in \mathfrak E_{<\infty}$ implies $G_p\in \mathfrak E_{<\infty}$ for every $p\in\Prm$.

The last assertion follows from the fact that under the assumption $G = \varpi(G)$ the subgroup $B(G)$ is a direct summands of $G$, so Lemma 
\ref{Laaaast:lemma} applies. 
\end{proof}

\subsection{Locally compact abelian $p$-groups in $\mathfrak{E}_{<\infty}$ and in $\mathfrak{E}_0$}

In the following example we see that $\mathfrak{E}_{<\infty}$ contains locally compact abelian $p$-groups $G$ with $\rank_p(G)$ infinite, and that the same abelian group $G$ may be endowed with two different locally compact topologies $\tau_1,\tau_2$ with $(G,\tau_1)\in\mathfrak E_0$ and $(G,\tau_2)\not\in\mathfrak E_{<\infty}$.

\begin{example}\label{FGR}
Any torsion abelian $p$-group of infinite rank (e.g., the group $\Z(p)^{\N}$) equipped with the discrete topology belongs to $\mathfrak{E}_0 \subseteq \mathfrak{E}_{<\infty}$ by Remark~\ref{discrete=0} and it has finite $p$-rank.

On the other hand, $\Z(p)^{\N}$ endowed with the compact product topology does not belong to $\mathfrak E_{<\infty}$. In fact, $\Z(p)^\N\cong(\Z(p)^\N)^\N$ and, letting $K=\Z(p)^\N$, the one-sided left Bernoulli shift $$\sigma:K^\N\to K^\N,\ (x_0, x_1, x_2, \ldots)\mapsto (x_1, x_2, x_3, \ldots),$$ has infinite topological entropy \cite{DG-islam,DS}. A similar argument shows that the compact group $\Z(p)^{\kappa}$ belongs to $\mathfrak E_{<\infty}$ for no infinite cardinal $\kappa$ (just note that $\Z(p)^\kappa\cong(\Z(p)^\kappa)^\N$). 
\end{example}

In the following proposition we see how one can compute the topological entropy of a continuous endomorphism of a locally compact abelian $p$-group $G$ with $\rank_p(G)<\infty$.

Note that for $G=\Z_p^{n_1} \times \Q_p^{n_2} \times \Z(p^\infty)^{n_3} \times F_p$, with $n_1,n_2,n_3\in\N$ and $F_p$ a finite $p$-group, 
the subgroup 
$$d(G)=\Q_p^{n_2}\times \Z(p^\infty)^{n_3}$$ is fully invariant in $G$, hence the subgroup  $t(d(G)))=\Z(p^\infty)^{n_3}$ is fully invariant in $G$.
Moreover,  $$d(G)/t(d(G))\cong d(G/t(G))\cong\Q_p^{n_2}$$ and for elements of $\End(\Q_p^{n_2})$ the topological entropy can be explicitly computed by applying Theorem~\ref{yuzvinskip}.

\begin{proposition}\label{particularbutgeneral}
Let $G$ be a locally compact abelian $p$-group with $\rank_p(G)<\infty$ and $\phi\in\End(G)$. Then $$h_{top}(\phi)=h_{top}\left(\bar\phi_{G/t(G)}\restriction_{d(G/t(G))}\right).$$
\end{proposition}
\begin{proof}
By Theorem~\ref{p-rank-finite} we can assume that $G=\Z_p^{n_1}\times\Q_p^{n_2}\times \Z(p^\infty)^{n_3}\times F_p$, for some $n_1,n_2,n_3\in\N$ and a finite $p$-group $F_p$.

Since $t(G)=\Z(p^\infty)^{n_3}\times F_p$, Corollary~\ref{cor:fundamentalstep} entails
\begin{equation}\label{eq3}
h_{top}(\phi)=h_{top}(\bar\phi_{G/t(G)}), 
\end{equation}
where $\bar\phi_{G/t(G)}\in\End(G/t(G))$ is induced by $\phi$ and $G/t(G)\cong\Z_p^{n_1}\times\Q_p^{n_2}$.
For $$H=G/t(G)\cong\Z_p^{n_1}\times \Q_p^{n_2}$$ note that $$D(H)\cong\Q_p^{n_1}\times\Q_p^{n_2}.$$ 
Since the subgroup $d(H)\cong\Q_p^{n_2}$ is fully invariant in $H$, $\bar\phi_{G/t(G)}$ is conjugated to
a $\psi\in\End(\Z_p^{n_1}\times \Q_p^{n_2})$ such that, for every $(x,y)\in\Z_p^{n_1}\times \Q_p^{n_2}$,
$$\psi(x,y)=(\psi_1(x),\psi_2(y)+\psi_3(x)),$$  
where $\psi_1\in\End(\Z_p^{n_1})$, $\psi_2\in\End(\Q_p^{n_2})$ and $\psi_3:\Z_p^{n_1}\to\Q_p^{n_2}$ is a continuous homomorphism.
Note that $\psi_2$ is conjugated to $\bar\phi_{H}\restriction_{d(H)}$.
By Lemma~\ref{conju}(b), 
\begin{equation}\label{eq2}
h_{top}(\psi)=h_{top}(\bar\phi_{H})\quad \text{and}\quad h_{top}(\psi_2)=h_{top}(\bar\phi_{H}\restriction_{d(H)}).
\end{equation}

Since the subgroup $\Z_p^{n_1}$ is open in $\Q_p^{n_1}$, Lemma~\ref{new:lemma1} gives that $\psi_1$ extends to $\tilde\psi_1\in\End(\Q_p^{n_1})$, $\psi$ extends to $\tilde \psi\in\End(\Q_p^{n_1}\times\Q_p^{n_2})$, 
\begin{equation}\label{eq1}
h_{top}(\psi_1)=h_{top}(\tilde\psi_1)\quad\text{and}\quad h_{top}(\psi)=h_{top}(\tilde\psi).
\end{equation}
Since $\Q_p^{n_2}$ is $\tilde\psi$-invariant, $\tilde\psi\restriction_{\Q_p^{n_2}}=\psi_2$ and the endomorphism induced by $\tilde \psi$ on $\Q_p^{n_1}\times\Q_p^{n_2}/\Q_p^{n_2}$ is conjugated to $\tilde\psi_1$, by Theorem~\ref{ATQp} and Lemma~\ref{conju}(b), we deduce that
$$h_{top}(\tilde\psi)=h_{top}(\psi_2)+h_{top}(\tilde\psi_1).$$ 
Since $h_{top}(\psi_1)=0$ by Corollary~\ref{invbase}, from \eqref{eq3}, \eqref{eq2} and \eqref{eq1} we have that $$h_{top}(\phi)=h_{top}(\psi)=h_{top}(\tilde\psi)=h_{top}(\psi_2)=h_{top}(\bar\phi_H\restriction_{d(H)}),$$
as required.
\end{proof}

We are ready to prove Theorem~\ref{main1}.

\begin{proof}[\bf Proof of Theorem~\ref{main1}]
By Theorem~\ref{p-rank-finite} we can assume that $G=\Z_p^{n_1}\times \Q_p^{n_2}\times \Z(p^\infty)^{n_3}\times F_p$, for some $n_1,n_2,n_3\in\N$ and a finite $p$-group $F_p$.

By Proposition~\ref{particularbutgeneral}, for $\phi\in\End(G)$, we obtain $h_{top}(\phi)=h_{top}\left(\bar\phi_{G/t(G)}\restriction_{d(G/t(G))}\right)$.
This value is finite since $d(G/t(G))\cong\Q_p^{n_2}$ and $\Q_p^{n_2}\in\mathfrak E_{<\infty}$ by Theorem~\ref{yuzvinskip}. Hence, $G\in\mathfrak E_{<\infty}$.

\smallskip
We verify now that $G\in \mathfrak E_0$ precisely when $n_2=0$.
If $n_2 \neq 0$, then $G\not\in\mathfrak E_0$ because $\Q_p^{n_2}\not\in\mathfrak E_0$ by Example~\ref{Qpp} or by Theorem~\ref{yuzvinskip}.
Assume that $n_2=0$, so $$G=\Z_p^{n_1}\times\Z(p^\infty)^{n_3}\times F_p.$$ 
By Remark~\ref{discrete=0} we have that $t(G)=\Z(p^\infty)^{n_3}\times F_p\in\mathfrak E_0$, while $\Z_p^{n_1}\in\mathfrak E_0$ by Corollary~\ref{invbase}. Since $AT(G,t(G))$ holds in this case by Corollary~\ref{cor:fundamentalstep}, we conclude that $G\in\mathfrak E_0$ by Lemma~\ref{Lemma3space}.
\end{proof}

Combining Proposition~\ref{proofMain2}  with Theorem~\ref{main1} we obtain the proof of Theorem~\ref{main2}:

\begin{proof}[\bf Proof of Theorem~\ref{main2}]
Let $G$ be a locally compact abelian with $c(B(G))=\{0\}$ and $\mathrm{rank}_p(G_p)<\infty$ for every $p\in\Prm$.

(a) If $\sum_{p\in \Prm} \rk_{\Z_p}(d(G_p))< \infty$, then $\rk_{\Z_p}(d(G_p))<\infty$  for every $p\in\Prm$ and $\rk_{\Z_p}(d(G_p))=0$ for almost all $p\in\Prm$.
By Theorem~\ref{main1}, this means that $G_p\in \mathfrak E_{<\infty}$  for every $p\in\Prm$ and $G_p \in \mathfrak E_0$ for almost all $p\in\Prm$, and so $G \in \mathfrak E_{<\infty}$ by Proposition~\ref{proofMain2}. 
 
(b) If $c(G)=\{0\}$ and $\rk_{\Z_p}(d(G_p)) = 0$ for every $p\in\Prm$, then $G_p\in\mathfrak E_0$ for every $p\in\Prm$ by Theorem~\ref{main1}, and so $G\in\mathfrak E_{0}$ by Proposition~\ref{proofMain2}.

Analogously one can prove the converse implications under the assumption $G=\varpi(G)$.
\end{proof}

\section{The scale}\label{scalesec}

If $G$ is a totally disconnected locally compact group and $\phi\in\End(G)$, the \emph{scale} of $\phi$ was defined in \cite{Willis} as
$$s(\phi)=\min\{[\phi(U):U\cap\phi(U)]\mid U\in\B(G)\}.$$
By \cite[Proposition 4.8]{GBV}, we have that always
\begin{equation}\label{s<h}
\log s(\phi)\leq h_{top}(\phi).
\end{equation}

A subgroup $U\in\B(G)$ is \emph{minimizing} for $s(\phi)$ if the minimum in the definition is attained at $U$, that is, $s(\phi)=[\phi(U):U\cap\phi(U)]$. 
Obviously, every $\phi$-invariant subgroup $ U\in\B(G)$ is minimizing and witnesses the equality $s(\phi)=1$; in particular,  
$s(\phi)=1$ for every $\phi \in \End(G)$, if $G$ is either compact or discrete. 

The \emph{nub} $\nub(\phi)$ of $\phi$ is the intersection of all minimizing subgroups of $s(\phi)$; by \cite[Corollary 4.6 and Corollary 4.11]{GBV}  $\nub(\phi)$ is a compact subgroup of $G$ such that $\phi(\nub(\phi))=\nub(\phi)$
and
$$h_{top}(\phi)=\log s(\phi)\ \text{if and only if}\ \nub(\phi)=\{1\}.$$
 In \cite[Theorem 3.32]{UdoW} several conditions equivalent to $\nub(\phi)=\{1\}$ are given in case $\phi$ is a topological automorphism.

\medskip
The scale can be computed also by using the following useful formula, called M\"oller's formula.

\begin{fact}[{See \cite[Proposition 18]{Willis}}]
Let $G$ be a totally disconnected locally compact group and $\phi\in\End(G)$. If $U\in\B(G)$, then
$$s(\phi)=\lim_{n\to\infty}[\phi^n(U):U\cap \phi^n(U)]^{\frac{1}{n}}.$$
\end{fact}

In the abelian case $\phi^n(U)/(U\cap \phi^n(U))\cong (U+\phi^n(U))/U$, hence M\"oller's formula can be written as
\begin{equation}\label{MF}
s(\phi)=\lim_{n\to\infty}\left|\frac{U+\phi^n(U)}{U}\right|^{\frac{1}{n}}.
\end{equation}

The counterpart of the Addition Theorem, namely $s(\phi)=s(\phi\!\restriction_H)s(\bar\phi_{G/H})$, does not hold for the scale (see \cite[Remark 4.6]{BDG}). 
Nevertheless, by applying the formula in \eqref{MF} we can easily extend the following monotonicity needed below to all continuous endomorphisms. 

\begin{lemma}\label{monots}
Let $G$ be a totally disconnected locally compact abelian group, $\phi\in\End(G)$ and $H$ a $\phi$-invariant closed subgroup of $G$. Then 
$s(\phi)\geq\max\{s(\phi\restriction_H),s(\bar\phi_{G/H})\}$.
\end{lemma}
\begin{proof}
By \cite[Corollary 2.5]{GBV}, $\B(H)=\{U\cap H\mid U\in\B(G)\}$ and $\pi(\B(G))=\{\pi(U)\mid U\in\B(G)\}\subseteq\B(G/H)$ is a local base of $G/H$.

Let $U\cap H\in\B(H)$. Then, for every $n\in\N$, 
\begin{align*}
\frac{U+\phi^n(U)}{U}\geq \frac{U+\phi^n(U\cap H)}{U}&=\frac{U+(U\cap H)+\phi^n(U\cap H)}{U}\\&\cong\frac{(U\cap H)+\phi^n(U\cap H)}{U\cap((U\cap H)+\phi^n(U\cap H))}=\frac{(U\cap H)+\phi^n(U\cap H)}{U\cap H}.
\end{align*}
To conclude that $s(\phi\restriction_H)\leq s(\phi)$, apply \eqref{MF}.

Let now $\pi(U)\in\pi(\B(G))$. Then $U+\phi^n(U)/U$ has as a quotient
$$\frac{U+\phi^n(U)}{(U+H)\cap(U+\phi^n(U))}\cong\frac{U+\phi^n(U)+H}{U+H}\cong\frac{\pi(U+\phi^n(U)+H)}{\pi(U+H)}=\frac{\pi(U+\phi^n(U)}{\pi(U)}=\frac{\pi(U)+(\bar\phi_{G/H})^n(\pi(U))}{\pi(U)}.$$
To conclude that $s(\bar\phi_{G/H})\leq s(\phi)$, apply \eqref{MF}.
\end{proof}

The $p$-adic Yuzvinski's formula for the scale was given for topological automorphisms of $\Q_p^n$ in \cite[Theorem 5.2]{BDG} (see also \cite{Gl} for more general results) and can be generalized to continuous endomorphisms applying the same argument.

\begin{theorem}\label{yuzs}
Let $n\in\N$ and $\phi\in\End(\Q_p^n)$. Then $s(\phi)=\prod_{|\lambda_i|_p>1}|\lambda_i|_p$,
where $\{\lambda_i\mid i\in\{1,\ldots,n\}\}$ are the eigenvalues of $\phi$ in some finite extension $K_p$ of $\Q_p$.
\end{theorem}

From Theorem~\ref{yuzvinskip} and Theorem~\ref{yuzs}, we get immediately the following equality.

\begin{corollary}\label{h=sp}
Let $n\in\N$ and $\phi\in\End(\Q_p^n)$. Then $h_{top}(\phi)=\log s(\phi)$.
\end{corollary}

We extend Corollary~\ref{h=sp} to all locally compact abelian $p$-groups of finite $\rank_p$ in the next result which covers Theorem~\ref{main4}.

\begin{theorem}\label{h=srp}
Let $G$ be a locally compact abelian $p$-group with $\rank_p(G)<\infty$ and let $\phi\in\End(G)$. Then $h_{top}(\phi)=\log s(\phi)$ and $\nub(\phi)=\{0\}$.
\end{theorem}
\begin{proof}
By Theorem~\ref{p-rank-finite} we can assume that $G=\Z_p^{n_1}\times\Q_p^{n_2}\times \Z(p^\infty)^{n_3}\times F_p$, for some $n_1,n_2,n_3\in\N$ and a finite $p$-group $F_p$.
By Proposition~\ref{particularbutgeneral},
$h_{top}(\phi)=h_{top}(\bar\phi_{G/t(G)}\restriction_{d(G/t(G))})$, where $d(G/t(G))\cong\Q_p^{n_2}$.
By Lemma~\ref{monots} and Corollary~\ref{h=sp}, $$\log s(\phi)\geq \log s(\bar\phi_{G/t(G)}\restriction_{d(G/t(G))})=h_{top}(\bar\phi_{G/t(G)}\restriction_{d(G/t(G))}).$$ 
Since $\log s(\phi)\leq h_{top}(\phi)$ by \eqref{s<h}, we obtain the thesis.
\end{proof}

\section{Final comments and open questions}\label{Hsec}

\subsection{The abelian case}

We start with a question related to Lemma~\ref{application}.

\begin{question}\label{notaware}
Let $G$ be a locally compact abelian group and endow $D(G)$ with the unique group topology that makes $G$ an open topological subgroup of $D(G)$. 
\begin{enumerate}[(a)]
\item Does $G \in \mathfrak{E}_{<\infty}$ imply $D(G) \in \mathfrak{E}_{<\infty}$?
\item Does $G \in \mathfrak{E}_0$ imply $D(G) \in \mathfrak{E}_0$?
\end{enumerate}
\end{question}

According to Theorem~\ref{origin}(a), a locally compact abelian group $G\in \mathfrak E_{<\infty}$ is  finite-dimensional. This motivates us to focus on finite-dimensional locally compact abelian groups in the sequel.

Theorem~\ref{red1} leaves open the following questions. A positive answer to both items would completely reduce the problem of understanding the structure of the locally compact abelian groups in $\mathfrak E_{<\infty}$ to the totally disconnected case. 

\begin{question}\label{ques0}
Suppose that $G$ is a finite-dimensional locally compact abelian group. 
\begin{enumerate}[(a)]
\item Does $G/c(G) \in \mathfrak E_{<\infty}$ imply $G \in \mathfrak E_{<\infty}$?
\item Does $G \in \mathfrak E_{<\infty}$ imply $G/c(G) \in \mathfrak E_{<\infty}$?
\end{enumerate}
\end{question}

According to Theorem~\ref{origin}(a), the answer to item (a) is affirmative for compact abelian groups.
On the other hand, the answer to item (b) is not known even for compact abelian groups (see \cite[Question 7.3]{DS}). 

\medskip
We do not know whether the implication in Lemma~\ref{last?} can be inverted: 

\begin{question}\label{ques07}
Does $G \in \mathfrak E_{<\infty}$ imply $\varpi (G) \in \mathfrak E_{<\infty}$ for a  locally compact abelian group $G$?
\end{question}

A positive answer to this question would allow us to work in the case when $G=\varpi(G)$.
Assuming also that $c(B(G))=\{0\}$ (i.e., $B(G)$ is periodic), in Proposition~\ref{proofMain2} we saw that $G\in\mathfrak E_{<\infty}$ if and only if $B(G) \in \mathfrak E_{<\infty}$. 
So the study of the locally compact abelian groups in $\mathfrak E_{<\infty}$ would be reduced to the case of periodic locally compact abelian groups, for which Theorem~\ref{prp:prod} gives a further reduction to locally compact abelian $p$-groups.

\smallskip
Question~\ref{ques07} remains open even in the totally disconnected case, when $\varpi(G) = B(G)$:

\begin{question}\label{ques00}
Suppose that $G$ is a totally disconnected locally compact abelian group. Does $G \in \mathfrak E_{<\infty}$ imply $B(G) \in \mathfrak E_{<\infty}$?
\end{question}

The answer is positive whenever $G$ is divisible, since in this case every continuous endomorphism of $B(G)$ extends to an endomorphism of $G$, as $G$ is divisible and $B(G)$ is an open subgroup of $G$.  

\medskip
Theorem~\ref{ATQp}, stating that the Addition Theorem holds for $\Q_p^n$, leaves open the following general question.

\begin{question}\label{ques1}
Does $AT(G)$ hold for every locally compact abelian $p$-group $G$? What about locally compact abelian $p$-groups with $\mathrm{rank}_p(G)<\infty$?   
\end{question}

According to Theorem~\ref{AT:AT}, an affirmative answer to the first part of the above question would imply that $AT(G)$ holds for every totally disconnected locally compact abelian group.
We conjecture that the answer is affirmative at least in the second more restrictive version.

\subsection{The non-abelian case}

Finally, we report a few comments regarding the non-abelian case. 

Following \cite{lg}, for a compact $p$-group $G$ and $n \in \N$, consider the subgroup $$\Omega_n(G)=\overline{\langle \{ g\mid g^{p^n}=1 \} \rangle}, $$ which is obviously fully invariant in $G$. 
\begin{lemma}\label{Lemma:heis}
If $G$ is a compact $p$-group such that $\mathcal B=\{\Omega^n(G)\mid n\in\N\}\subseteq \B(G)$ is a local base of $G$, then $G\in \mathfrak{E}_0$. 
\end{lemma}
\begin{proof}
It suffices to apply Corollary~\ref{invbase}.
\end{proof}

An instance of the above lemma, are the groups $\Z_p^n\times F_p$, where $n\in \N$ and $F_p$ is a finite $p$-group.

\medskip
In order to provide another example we need to recall first a family of non-abelian locally compact nilpotent groups of nilpotency class $2$ with significant applications in theoretical physics and Lie theory.  
The \textit{Heisenberg group} on a commutative unitary ring $R$ is the group of all $(3\times 3)$-matrices
$$\mathbb{H}(R)=\{M(a,b;z) \mid a,b,z \in R\},\quad \text{where}\ M(a,b;z)=\begin{pmatrix}1 & a & z \\ 0 & 1 & b \\ 0 & 0 & 1 \end{pmatrix}.$$
The group $\mathbb{H}(R)$ is nilpotent of class $2$, since 
$$Z(\mathbb{H}(R))=\mathbb{H}(R)'=\{M(0,0;z)\mid z\in R\} \cong (R,+)\quad\text{and}\quad\mathbb{H}(R)/Z(\mathbb{H}(R)) \cong (R,+)\times (R,+).$$ 
When $R$ is a topological commutative unitary ring, $\mathbb{H}(R)$ is equipped with the product topology induced by $R^9$.  

\smallskip
One can see that  $\mathbb{H}(\Z_p)$ is a compact $p$-group, while $\mathbb{H}(\Q_p)$ is a locally compact $p$-group; for both these groups $\rank_p$ equals $2$. Here, following \cite{HHR}, we take as a definition of $\rank_p(G)$ of a locally compact group $G$, the equivalent condition mentioned in the introduction, that is, $\rank_p(G)$ is finite if there exists $n\in\N$ such that every topologically finitely generated subgroup of $G$ is generated by at most $n$ elements, and $\mathrm{rank}_p(G)$ is the smallest $n\in\N$ with such property.

\begin{example}
The group $\mathbb H(\Z_p)$ satisfies the hypothesis of Lemma~\ref{Lemma:heis}, hence $\mathbb H(\Z_p)\in\mathfrak E_0$. 
\end{example}

While $\mathbb H(\Z_p)\in\mathfrak E_0$ comes as an application of Lemma~\ref{Lemma:heis}, for $\mathbb H(\Q_p)$ we prove the following result using several versions of the Addition Theorem and specific algebraic properties of the Heisenberg groups.

\begin{theorem}
The group $G= \mathbb H(\Q_p)$ satisfies $G\in\mathfrak E_{<\infty}\setminus\mathfrak E_0$.
\end{theorem}
\begin{proof}
To check that $G\not\in\mathfrak E_0$, let $$\phi:G\to G,\ M(a,b;z)\mapsto M(a/p,b;z/p).$$ Then $\phi\in \Aut(G)$, so by the Addition Theorem for topological automorphisms of totally disconnected locally compact groups in \cite{GBV}, 
\begin{equation}\label{AT-}
h_{top}(\phi)= h_{top}(\phi\restriction_{Z(G)}) +   h_{top}(\bar \phi_{G/Z(G)}).
\end{equation} 
Since
\begin{equation}\label{AT--}
Z(G)\cong \Q_p\quad\text{and}\quad G/Z(G) \cong \Q^2_p,
\end{equation} 
and $\phi\restriction_{Z(G)}$ is conjugated to the multiplication by $1/p$ of $\Q_p$, by Lemma~\ref{conju}(b) and Example~\ref{Qpp}
we conclude that $h_{top}(\phi)\geq h_{top}(\phi\restriction_{Z(G)})= \log p >0$.

\smallskip
It remains to verify that $G\in\mathfrak E_{<\infty}$. To this end, fix $\phi\in\End(G)$ and let $N=\ker\phi$; we have to prove that $h_{top}(\phi)<\infty$.

Assume that $N=\{1\}$. We show that $\phi\in\Aut(G)$. Indeed, since $Z(G)$ is fully invariant, $\phi\restriction_{Z(G)}$ is an injective endomorphism of $Z(G) \cong \Q_p$, hence $\phi\restriction_{Z(G)}$ is an automorphism of $Z(G)$. In particular, $\phi^{-1}(Z(G)) = Z(G)$, so the induced endomorphism $\bar\phi_{G/Z(G)} : G/Z(G) \to  G/Z(G)$ is injective as well. 
As $G/Z(G)\cong \Q_p\times \Q_p$, $\bar \phi_{G/Z(G)}$ is an automorphism of $G/Z(G)$. This proves  that $\phi$ is a bijective continuous  endomorphism of $G$. 
Since $\mathbb{H}(\Q_p)$ is locally compact and $\sigma$-compact, $\phi$ is a topological automorphism by the Open Mapping Theorem (see \cite[Theorem 5.29]{HR}). 
In view of \eqref{AT--}, by Lemma~\ref{conju}(b) and by Theorem~\ref{main1}, $h_{top}(\phi\restriction_{Z(G)})<\infty$ and $h_{top}(\bar \phi_{G/Z(G)})< \infty$, 
so we can conclude with \eqref{AT-} that $h_{top}(\phi)<\infty$. 

Now suppose that $N\neq\{1\}$. First we show that $H=N\cap Z(G)\neq\{1\}$ and then the inclusion $Z(G)\subseteq N$. Indeed, if $H=N\subseteq Z(G)$, there is nothing to prove, as  $N\neq\{1\}$. If there exists a non-central element $y\in N$, then there exists $x\in G$ such that the commutator $[x,y]\ne 1$. This implies $H=N\cap G' \ne \{1\}$, as $[x,y]\in N$. 
Then $H$ is a non-trivial closed subgroup of $Z(G)$, hence $Z(G)/H$ is torsion in view of \eqref{AT--}. On the other hand, $Z(G)/H \cong \phi(Z(G)) $ is (isomorphic to) a subgroup of $G$, hence torsion-free as $G$ itself. Consequently $Z(G)/H$ is trivial, that is $H=Z(G)\subseteq N$.

Since $N$ contains $Z(G)=G'$, Theorem \ref{LAAAAAst:thm} is applicable and gives $ h_{top}(\phi) = h_{top}(\bar \phi_{G/N})$. Since $G/N$ is 
isomorphic to a quotient of $G/Z(G)$, we get $G/N\in\mathfrak E_{<\infty}$ by \eqref{AT--} and in view of Theorem~\ref{main1}.
Therefore, $h_{top}(\phi)=h_{top}(\bar \phi_{G/N})<\infty$.
\end{proof}

This theorem motivates the following:

\begin{conjecture}\label{C0nj}
Every nilpotent locally compact $p$-group $G$ with $\rank_p(G)<\infty$ is in $\mathfrak E_{<\infty}$.
\end{conjecture}

Due to the Addition Theorem for topological automorphisms of totally disconnected locally compact groups from \cite{GBV}, it is not hard to deduce from Theorem~\ref{main1} an affirmative answer of the above conjecture for topological automorphisms. 

\medskip
Since totally disconnected locally compact groups are precisely the zero-dimensional ones, now we impose again finiteness of the dimension, instead of finiteness of the rank.  We conjecture that the following question has a positive answer in the abelian case. This is inspired by the implication in the first part of Theorem~\ref{origin}(a).

\begin{question}\label{ques:conn}
If $G$ is a finite-dimensional connected locally compact group, can we assert that $G \in\mathfrak E_{<\infty}$?   
\end{question}

According to \cite[Corollary 16]{B}, the answer is affirmative for Lie groups. On the other hand, it is affirmative also for compact groups. Indeed, the second part of Theorem~\ref{origin}(a) (this is \cite[Theorem A]{DS}) works also for classes of compact-like groups other than locally compact (e.g., $\omega$-bounded, or simply pseudocompact, etc.). In particular, $G\in \mathfrak E_{<\infty}$ if $G$ is a finite-dimensional connected pseudocompact group by \cite[Proposition 4.4]{DS}. 
So, one can try to further push the study of the class $\mathfrak E_{<\infty}$ in the framework that simultaneously generalizes both locally compact and pseudocompact groups, namely, that of locally pseudocompact groups.

\bibliographystyle{plain}

\end{document}